\documentclass{amsart}

\usepackage{amscd,amsmath,amssymb,amsthm,mathrsfs,latexsym}
\usepackage[mathcal]{euscript}
\usepackage{mathrsfs}

\usepackage[all]{xy}
\usepackage{color,soul} 
\usepackage{hyperref} 
\hypersetup{
 colorlinks=true, 	 
 linktoc=all,     	 
 linkcolor=blue,  	 
 citecolor=blue,    
 filecolor=magenta,  
 urlcolor=black       
}

\usepackage{tikz-cd}

\usepackage{enumerate}

\newtheorem{thm}{Theorem}[section]
\newtheorem{cor}[thm]{Corollary}

\newtheorem{prop}[thm]{Proposition}

\theoremstyle{definition}

\newtheorem{rmk}[thm]{Remark}
\newtheorem*{rmkMT}{Remark on Molien's Theorem}
\newtheorem*{rmkH1}{Remark on the rank of the fundamental group}
\theoremstyle{definition}
\newtheorem{ex}[thm]{Example}
\theoremstyle{remark}

\newtheorem{question}[thm]{Question}

\newtheoremstyle{named}{}{}{\itshape}{}{\bfseries}{.}{.5em}{\thmnote{#3's }#1}
\theoremstyle{named}

\newcommand{\srm}[1]{\stackrel{#1}{\maps}}

\newcommand{\e}{\emph}
\def\co{\colon\thinspace}

\newcommand{\R}{{\mathbb R}}
\newcommand{\C}{{\mathbb C}}
\newcommand{\Z}{{\mathbb Z}}

\newcommand{\T}{{\mathcal T}}
\newcommand{\X}{{\rm X}}

\newcommand{\Q}{{\mathbb Q}}
\newcommand{\F}{{\mathbb F}}

\newcommand{\Susp}{\Sigma}

\newcommand{\bbZ}{\mathbb{Z}}

\newcommand{\bbQ}{\mathbb{Q}}

\newcommand{\rk}{{\rm rank\,}}
\newcommand{\Tr}{{\rm trace\,}}

\newcommand{\Hom}{{\rm Hom}}
\newcommand{\Comm}{{\rm Comm}}

\newcommand{\leqs}{\leqslant}
\newcommand{\geqs}{\geqslant}
\newcommand{\heq}{\simeq}

\newcommand{\maps}{\longrightarrow}

\newcommand{\injects}{\hookrightarrow}

\newcommand{\isom}{\cong}
\newcommand{\cross}{\times}

\def\ds{\displaystyle}


\title[Poincar\'e series for spaces of commuting elements in Lie groups]{Hilbert--Poincar\'e series for spaces of commuting elements in Lie groups}

\author{Daniel A. Ramras}
\address{Indiana University - Purdue University Indianapolis, Indianapolis, IN 46202}
\email{dramras@iupui.edu}

\author{Mentor Stafa}
\address{Tulane University, New Orleans, LA 70118}
\email{mstafa@tulane.edu}

\date{\today}
\thanks{D. Ramras was partially supported by a Simons Collaboration grant (\#279007)}

\begin{document}

\begin{abstract}
In this article we study the homology of spaces $\Hom(\Z^n,G)$ of ordered pairwise commuting $n$-tuples in a Lie group $G$. We give an explicit formula for the Poincar\'e series of these spaces in terms of invariants of the Weyl group of $G$. By work of Bergeron and Silberman, our results also apply to $\Hom(F_n/\Gamma_n^m,G)$, where the subgroups $\Gamma_n^m$ are the terms in the descending central series of the free group $F_n$. Finally, we show that there is a stable equivalence between the space $\Comm(G)$ studied by Cohen--Stafa and its nilpotent analogues.
\end{abstract}

\maketitle

\tableofcontents

\section{Introduction}

Let $G$ be a compact and connected Lie group and let $\pi $
be a discrete group generated by $n$ elements.
In this article we study the rational  homology of the space of
group homomorphisms $\Hom(\pi,G)\subseteq G^n$, endowed with the
subspace topology from $G^n$. In particular, 
when $\pi$ is free abelian or nilpotent 
we give an explicit formula for
the Poincar\'e series of $\Hom(\pi,G)_1$, the connected component of the trivial representation,
in terms of invariants of the Weyl group $W$ of $G$.

The topology of the spaces $\Hom(\pi,G)$ has been studied extensively in recent years,
in particular when $\pi$ is a free abelian group
\cite{adem2007commuting,bairdcohomology,BJS,gomez.pettet.souto,pettet.souto,stafa.comm,stafa.comm.2};
in this case $\Hom(\bbZ^n, G)$ is known as
\textit{the space of ordered commuting $n$-tuples in $G$}.
The case in which $\pi$ is a finitely generated
nilpotent group was recently analyzed by Bergeron and Silberman~\cite{bergeron,bergeron2016note}.
These spaces and variations thereon, such as the space of
almost commuting elements \cite{borel2002almost},
have been studied in various settings,  
 including
work of  Witten and Kac--Smilga  on supersymmetric Yang-Mills theory 
\cite{witten1,witten2,kac.smilga}.

Our formula for the Poincar\'e series of the identity component $\Hom(\Z^n,G)_1$ builds on 
work of Baird~\cite{bairdcohomology} and Cohen--Reiner--Stafa~\cite{stafa.comm}.
In fact, we give a formula for a more refined \textit{Hilbert--Poincar\'e series}, 
which is a tri-graded version of the standard Poincar\'e series that arises from a certain cohomological description of these spaces due to Baird.
Work  
of Bergeron and Silberman \cite{bergeron2016note}
then leads immediately to
 results for nilpotent groups.  
 The formula we produce is obtained by comparing  stable splittings of  
 $\Hom(\Z^n,G)$ and of the space $\Comm (G)$ introduced in~\cite{stafa.comm}. 
The latter space is an analogue of the James reduced product construction for 
commuting elements in $G$; see Section~\ref{Comm-sec}.

\subsection{Main results} The main purpose of this paper is to give
an explicit formula for
the Poincar\'e series of the
 component $\Hom(\Z^n,G)_1$.
 
\begin{thm}\label{thm: Poincare series of Hom INTRO}
The  Poincar\'e series of $\Hom(\Z^n,G)_1$ is given by
$$
P(\Hom(\Z^n,G)_1;q)=\frac{\prod_{i=1}^r (1-q^{2d_i})}{|W|} 
\left(\sum_{w\in W} \frac{ \det(1+qw)^n}{\det(1-q^2w)} \right),
$$
where the integers $d_1,\dots,d_r$ are the characteristic degrees of the Weyl group $W$.
\end{thm}

A similar formula for the homology of the character variety $\Hom(\Z^n, G)/G$ appears in Stafa~\cite{Stafa-char-var}.
 
Some comments are in order regarding the above formula. 
Let $T \subset G$ be a maximal torus with lie algebra $\mathfrak{t}$.
Then the Weyl group $W$ acts on the dual space $\mathfrak{t}^*$  as a  finite reflection group, and the determinants in the formula are defined in terms of this linear representation of $W$.
The characteristic  degrees of $W$ arise by  considering the induced action of $W$  on the polynomial algebra $\R[x_1, \ldots, x_r]$, where the $x_i$ form a basis for  $\mathfrak{t}^*$ (so $r =\rk (G)$). It is a theorem of Shephard--Todd~\cite{shephard1954finite} and Chevalley~\cite{chevalley1955invariants} that the $W$--invariants  $\R[x_1, \ldots, x_r]^W$ form a polynomial ring with $r$ homogeneous generators. The characteristic degrees of $W$ are then the degrees of the homogeneous generators for $\R[x_1, \ldots, x_r]^W$. These degrees are well-known, and are displayed in Table~\ref{table: characteristic degrees}.
For further discussion of these ideas, see Section~\ref{FRG}.

 The spaces $\Hom(\pi,G)$ are not path-connected in general;  
 see for instance 
\cite{giese.sjerve} where the path components of $\Hom(\Z^n,SO(3))$ are described.
However, if there is only one conjugacy class of maximal abelian subgroups in $G$,
namely the conjugacy class of maximal tori,
then $\Hom(\Z^n,G)$ and $\Comm(G)$ are both path-connected.
 This is true, for instance, if $G=U(n)$, $SU(n)$, or $Sp(n)$;
on the other hand, $\Hom(\Z^n,SO(2n+1))$ is disconnected for $n\geqs 2$ and 
$\Hom(\Z^n,G_2)$ is disconnected for $n\geqs 3$.
In fact, Kac and Smilga have classified those compact, simple Lie groups for which $\Hom(\Z^n, G)$ is path connected~\cite{kac.smilga}.
Moreover, when $G$ is semisimple and simply connected, it is a theorem of Richardson that $\Hom(\Z^2, G)$ is an  
irreducible algebraic variety, and hence is connected \cite{richardson1979commuting}.

Theorem \ref{thm: Poincare series of Hom INTRO} can also be applied to nilpotent groups.
Let $F_n \unrhd \Gamma^2_n \unrhd \Gamma^3_n \cdots$ be the descending
central series of the free group $F_n$. Bergeron and Silberman
\cite{bergeron2016note} show that for each $m\geqs 2$, the identity component $\Hom(F_n/\Gamma^m_n,G)_1$ 
consists entirely of abelian representations.
In fact, they show that if $N$ is a finitely generated nilpotent group, then the natural map
$$\Hom(N/[N,N], G)\longrightarrow \Hom(N, G)$$
restricts to a homeomorphism between the identity components.
Since $G$ admits a neighborhood $U$ of the identity that contains no subgroup other than the trivial subgroup, every representation in the identity component of $\Hom(N/[N,N], G)$ kills the torsion subgroup of $N/[N,N]$.
Thus our main result also yields the homology of the identity component in 
$\Hom(N, G)$.

The assumption that $G$ is compact is not in fact very restrictive, since if $G$ is 
the complex points of a connected
reductive linear algebraic group over $\C$ (or the real points if $G$ is defined over $\R$) 
and $K\leqs G$ is a maximal compact subgroup, then
Pettet and Souto \cite{pettet.souto} showed that $\Hom(\Z^n,G)$ deformation
retracts onto $\Hom(\Z^n,K)$. For simplicity, we refer to such groups $G$ simply as \e{reductive} Lie groups. Bergeron \cite{bergeron} generalized this result
to finitely generated nilpotent groups (in fact, Bergeron's result also allows $G$ to be disconnected). 
It should be emphasized, however, 
that for other discrete groups $\pi$ it is known 
that the homotopy types of $\Hom(\pi, G)$ and $\Hom(\pi, K)$ can differ; 
examples appear in~\cite{adem2007commuting}.

The above descending central series can be used to define
a filtration of the James reduced product of $G$, denoted $J(G)$,
which is also known as  
the free monoid generated by the based space $G$. The filtration
is given by the spaces
$$
\Comm(G)=\X(2,G)\subset \X(3,G) \subset \cdots \subset \X(\infty,G)=J(G)
$$
defined in Section \ref{sec: topology Hom}, and   
was studied
by Cohen and Stafa \cite{stafa.comm}. Here we show that all the
terms in the filtration have the same  Poincar\'e series.

\begin{thm}\label{thm: Poincare series of X(q,G) Intro}   
The inclusion
$$P(\Comm (G)_1;q)\injects P(\X (m, G)_1;q)$$
induces an isomorphism in homology for every $m\geq 2$ .
\end{thm}

\subsection{Structure of the paper}  
We start in Section \ref{sec: topology Hom} by giving
some basic topological properties of the spaces of 
homomorphisms $\Hom(\Z^n,G)$ and we define the spaces $\X (m, G) \subset J(G)$ considered above. In particular, we explain how all these spaces decompose into wedge sums 
after a single suspension. We prove Theorems 
\ref{thm: Poincare series of Hom INTRO} and 
\ref{thm: Poincare series of X(q,G) Intro} in Sections 
\ref{sec: Poincare series of Hom(Zn,G)} and 
\ref{sec: Poincare series of X(q,G)}, respectively. 
In Section~\ref{ungraded-sec}, we consider the ungraded cohomology and the rational complex $K$--theory of $\Hom(\Z^n, G)_1$.
Finally, we give examples of Hilbert--Poincar\'e series 
in Section \ref{sec: examples Poincare}, most
notably for the exceptional Lie group $G_2$.

\vspace{.2in}
\noindent {\bf Acknowledgements:} We thank Alejandro Adem and Fred Cohen for helpful comments, and Mark Ramras for pointing out the Binomial Theorem,
which simplified our formulas.

\section{Topology of commuting elements in Lie groups}\label{sec: topology Hom}

Let $G$ be a compact and connected Lie group. Fix a maximal torus $T\leqs G$ and let $W = N_G(T)/T$ be the Weyl group of
$G$. The map
\begin{equation}\label{c}G \times T \to G\end{equation}
 conjugating elements of the maximal
torus by elements of $G$ has been studied as far back as Weyl's work, and can be used to
show that the rational cohomology of $G$ is the ring of invariants
$[H^\ast(G/T) \otimes H^\ast(T)]^W$. 
To study the rational cohomology of $\Hom(\Z^n,G)$ we can proceed as follows.
The action by conjugation of $T$ on itself is trivial, so (\ref{c}) descends to a map
 $G/T\times T \to G$, which is invariant under the $W$--action $([g],t)\cdot [n] = ([gn],n^{-1}tn)$, where $n\in N_G (T)$. 
In~\cite{bairdcohomology} Baird showed that the induced map  
 \begin{align}\label{theta}
\begin{split}
\theta_n: G/T\times_{W} T^n &\to  \Hom(\Z^n,G)\\
[g,t_1,\dots,t_n] &\mapsto (gt_1g^{-1},\dots, gt_ng^{-1})
\end{split}
\end{align}
surjects 
onto  $\Hom(\Z^n,G)_1$ and
 induces an isomorphism of
rational cohomology groups
\begin{equation}\label{Baird}
H^\ast(\Hom(\Z^n,G)_1;\Q) \isom [H^\ast(G/T;\Q)\otimes H^\ast(T^n;\Q)]^W.
\end{equation}
This recovers the above fact about the cohomology of $G$ when $n=1$. 
Baird in  fact shows that all torsion in 
$H^\ast(\Hom(\Z^n,G)_1;\Z)$ has order
dividing $|W|$, but little else is known about the torsion in these spaces, 
beyond the case of $SU(2)$~\cite{BJS} and the fact that $H_1 (\Hom(\Z^n,G)_1; \Z)$ is torsion-free~\cite{gomez.pettet.souto}. 

We note that
as an ungraded $\Q W$-module, the ring $H^\ast(G/T;\Q)$ is simply
the regular representation $\Q W$, a well-known fact
that dates back to Borel \cite{borel1953cohomologie} --
a proof can be found for instance in the exposition by M. Reeder
\cite{reeder1995cohomology}. This fact implies that the
ungraded cohomology of the homomorphism space is just a
regraded version of the cohomology of $T^n$.
As we will see, various topological constructions related to the maps $\theta_n$ enjoy a similar structure in their cohomology.

Adem and Cohen \cite{adem2007commuting} showed that there is
a homotopy decomposition of the suspension of $\Hom(\Z^n,G)$
into a wedge sum of \textit{smaller} spaces as follows
\begin{equation}\label{eqn: stable decomp Hom}
\Sigma \Hom(\Z^n,G) \simeq \Sigma \bigvee_{1\leq k \leq n}
	 \bigvee_{n \choose k} \widehat{\Hom}(\Z^k,G),
\end{equation}
where $\widehat{\Hom}(\Z^k,G)$ is the quotient of $ \Hom(\Z^n,G)$ by the 
subspace consisting of all commuting $n$--tuples $(g_1, \ldots, g_n)$ 
such that $g_i=1$ for at least one coordinate $i$. 
This decomposition, along with the analogous decomposition of 
 $\Hom(\Z^n,G)_1$ given in Lemma~\ref{dec}, will play a key role in our study of homology.

Recall that the descending central series of a group $\pi$ is
the sequence of subgroups of $\pi$ given by
$
\pi=\Gamma^1 \unrhd \Gamma^2=[\pi,\pi] \unrhd
	\cdots \unrhd \Gamma^{k+1} \unrhd \cdots,
$
where inductively
$\Gamma^{k+1}=[\pi,\Gamma^{k}]$. Let $\Gamma^k_n$ be the $k$-th
stage in the descending central series of the free group $F_n$, and
note that $\Gamma_n^\infty = \bigcap_{k=1}^\infty \Gamma^k_n = 1$.
Then we obtain a filtration
\begin{equation}\label{eq: filtration of G^n}
\Hom(F_n/\Gamma^2_n,G)\subset \Hom(F_n/\Gamma^3_n,G)
		\subset \cdots \subset \Hom(F_n/\Gamma^\infty_n,G)=G^n
\end{equation}
of the space $G^n$ by subspaces of nilpotent $n$-tuples, where the first
term of the filtration is the space of commuting $n$-tuples
$F_n/\Gamma^2_n=F_n/[F_n,F_n]=\Z^n$.  It should be noted that
this filtration need not be exhaustive; that is, $\bigcup_{k=1}^\infty \Hom(F_n/\Gamma^k_n,G)$ is in general a proper subset of $G^n$.
Also note that
$T^n \subset F_n/\Gamma^2_n = \Hom(\Z^n,G)$, a fact that will be used later.
We obtain the following stable decompositions of the connected
components of the trivial representations for nilpotent $n$-tuples.

\begin{prop}\label{dec} Let $G$ be either a compact connected Lie group or a reductive connected Lie group.
For each  
 $m\geq 2$ there is a homotopy equivalence
\begin{align*}
\Sigma \Hom(F_n/\Gamma^m_n,G)_1 \simeq \Sigma
	\bigvee_{1\leq k \leq n} \bigvee_{n \choose k}
	\widehat{\Hom}(F_k/\Gamma^m_k,G)_1.
\end{align*}
In particular there is a homotopy equivalence
\begin{align*}
\Sigma \Hom(\Z^n,G)_1 \simeq \Sigma
	\bigvee_{1\leq k \leq n} \bigvee_{n \choose k}
	\widehat{\Hom}(\Z^k,G)_1.
\end{align*}

\end{prop}
\begin{proof} This is a minor modification to the arguments 
in~\cite[Corollary 2.21]{villarreal2016cosimplicial}, where the corresponding decompositions 
for the full representation spaces are obtained. 
The spaces $\{\Hom(F_n/\Gamma^m_n,G)\}_n$ form a simplicial space, 
which Villarreal shows is simplicially NDR in the sense defined 
in~\cite{adem.cohen.gitler.bahri.bendersky}. The face and degeneracy maps in 
these simplicial spaces preserve the identity components, so 
$\{\Hom(F_n/\Gamma^m_n,G)_1\}_n$ is also a simplicial space, and is 
again simplicially NDR. The decompositions now follow from the 
main result of~\cite[Theorem 1.6]{adem.cohen.gitler.bahri.bendersky}.
\end{proof}

\subsection{The James reduced product}

The \textit{James reduced product}
$J(Y)$ can be defined for any CW-complex $Y$ with basepoint $*$.
In our discussion $Y$ is usually a compact Lie group with basepoint
the identity element. Define $J(Y)$ as the quotient space
$$
J(Y):= \bigg( \bigsqcup_{n\geq 0} Y^n \bigg)/\sim
$$
where $\sim$ is the relation
$(\dots,*,\dots) \sim (\dots,\widehat{*},\dots)$
omitting the coordinates equal to the basepoint.
This can also be seen as the free monoid generated by the
elements of $Y$ with the basepoint acting as the identity element.
It is a classical result  
that $J(Y)$ is weakly homotopy equivalent
to $\Omega\Sigma Y$, the loops on the suspension of $Y$.
Moreover, the suspension of $J(Y)$ is given by
$$
\Sigma J(Y) \simeq \Sigma \bigvee_{n \geq 1} \widehat{Y^n},
$$
where $\widehat{Y^n}$ is the $n$-fold smash product.
It was first observed by Bott and Samelson \cite{bott1953pontryagin}
that the homology of $J(Y)$ is isomorphic as an algebra
to the tensor algebra $\T[\widetilde{H}_\ast(Y;R)]$ generated
by the reduced homology of $Y$, given that the homology of $Y$
is a free $R$-module. This is a central result used in our calculation.

\subsection{The spaces $\X(m,G)$}\label{Comm-sec}

Now consider the case in which $Y = G$, a connected Lie group 
with basepoint the identity element $1\in G$.
A filtration of the  
free monoid $J(G)$ is given by
\begin{equation}\label{eq: filtration of J(G)}
\X(2,G)\subset \X(3,G) \subset
\X(4,G) \subset \cdots \subset \X(\infty,G)=J(G),  
\end{equation}
where each space is defined by
$$
\X(m,G):=\bigg(\bigsqcup_{n\geq 0} \Hom(F_n/\Gamma^m_n,G) \bigg)/\sim
$$
where $\sim$ is the same relation as in $J(G)$. 
The spaces $\X(m,G)$ and $\Comm(G)=\X(2,G)$
were studied in \cite{stafa.comm}, where
 it was shown that   
$\Comm(G)$ carries important information 
about the spaces of commuting $n$-tuples $\Hom(\Z^n,G)$.
However, note that in general the spaces
$\X(m,G)$ do not have the structure of a monoid for any $m$.
As in the case of spaces of homomorphisms, the spaces $\X(m,G)$
need not be path connected. For instance, the space $\X(2,SO(3))$
has infinitely many path components, as shown in \cite{stafa.thesis}.
We can define the connected
component of the trivial representation for each space
$\X(m,G)$ by
$$\X(m,G)_1:=\bigg(\bigsqcup_{n\geq 0}
		\Hom(F_n/\Gamma^m_n,G)_1 \bigg)/\sim $$ 
with $\X(2,G)_1=\Comm(G)_1.$
Cohen and Stafa \cite[Theorem 5.2]{stafa.comm} show that
there is a stable decomposition of this space as follows:   
\begin{equation}\label{eqn: stable decomp Comm}
\Sigma \X(m,G) \simeq \Sigma
	\bigvee_{k \geq 1} \widehat{\Hom}(F_k/\Gamma^m_k,G).
\end{equation}

We have an analogous result for the identity components.

\begin{prop}\label{prop: X(q,G) decomposition}
Let $G$ be either a compact connected Lie group or a reductive connected Lie group.
For each 
$m\geq 2$ there is a homotopy equivalence
$$
\Sigma \X(m,G)_1 \simeq \Sigma
	\bigvee_{k \geq 1} \widehat{\Hom}(F_k/\Gamma^m_k,G)_1.$$
This is true in particular for $\Comm(G)_1$.
\end{prop}
\begin{proof}
This follows from the proof of \cite[Theorem 5.2]{stafa.comm}.
\end{proof}

\section{{Poincar\'e series of $\Hom(\Z^n,G)_1$}}\label{sec: Poincare series of Hom(Zn,G)}

For a topological space $X$ the (rational) \textit{Poincar\'e series} is
the series
$$
P(X; q):=\sum_{k\geq 0} \rk_\Q (H_k(X;\Q)) q^k.
$$
In this section we describe the Poincar\'e series of $\Hom(\Z^n,G)_1$.
Following \cite{stafa.comm}, we will refine the usual grading of cohomology 
and introduce  tri-graded \textit{Hilbert--Poincar\'e series}  for  
$X = \Hom(\Z^n,G)$, $\Comm(G)$, or $\X(m,G)$. 
These additional gradings will facilitate the computation 
of the Poincar\'e series itself.

For the remainder of this section, we will drop the coefficient 
group $\Q$ from our notation for (co)homology. The statements are true for any 
field with characteristic 0 or relatively prime to $|W|$.

The maps $\theta_{n}: G/T\times_{W} T^n \to \Hom(\Z^n,G)$ can be
assembled to give a map
\begin{equation}\label{Theta}
\Theta :  G/T\times_{W} J(T) \to \Comm(G)
\end{equation}
which surjects onto the connected component $\Comm(G)_1.$
It was shown in \cite{stafa.comm} that $\Theta$ induces
isomorphisms on the level of rational (co)homology, so   
rationally we obtain
\begin{equation}\label{Comm}
H^\ast(\Comm(G)_1) \cong [H^\ast(G/T)\otimes H^\ast(J(T))]^W
\cong  [H^\ast(G/T)\otimes \T^*[\widetilde{H}_*(T)] ]^W,
\end{equation}
where $\T^*$ denotes the dual of the tensor algebra.
This interpretation of the cohomology in terms of Weyl group invariants
allows us to make the following definition.
Define the Hilbert--Poincar\'e series of $\Comm(G)_1$ as the
tri-graded series
$$P(\Comm(G)_1;q,s,t)=\sum_{i,j,k \geq 0} \rk A(i,j,k)^W\,\, q^i s^j t^m,$$
where 
$$A(i,j,k) := H^i (G/T) \otimes \T^*[\widetilde{H}_*(T)]_{j,m}$$
and 
$\T^*[\widetilde{H}_*(T)]_{j,m}$
is the dual of the     
submodule of $\T[\widetilde{H}_*(T)]$   
generated by the $m$--fold tensors of total cohomological degree $j$.

To recover the ordinary Poincar\'e series
we can set $s$ equal to $q$ and $t$ equal to 1 since the tensor degree does not affect the (co)homological degree.
In order to understand this tri-graded version of the Poincar\'e series, we take a
short diversion to discuss the characteristic degrees of a finite reflection group.

\subsection{Finite reflection groups}\label{FRG}

A finite reflection group is a finite subgroup $W\subset GL_k(\mathbf{k})$, with  $\mathbf{k}$ a field of characteristic 0,
such that $W$ is  generated
by reflections. Equivalently, consider an $n$-dimensional vector
space $V$ over $\mathbf{k}$ equipped with the action of a
finite subgroup $W \subset GL (V)$. There is a corresponding
action on the symmetric algebra $R$ of $V$, which is isomorphic to the
polynomial algebra $R:=\mathbf{k}[x_1,\dots,x_n]$\footnote{Some authors (e.g. Grove and Benson~\cite[Chapter 7]{Grove-Benson}, or Humphreys~\cite[Chapter 3]{humphreys1992reflection}) replace $V$ by its dual $V^*$ in this discussion.}.
It is a classical result of Chevalley \cite{chevalley1955invariants} and
Shephard--Todd \cite{shephard1954finite} that when $W$ is generated by reflections, the invariant elements
of the $W$-action also form an algebra generated by $n$ elements,
and these generators can be chosen to be (algebraically independent) homogeneous
polynomials $f_1,\dots,f_n$. Hence the $W$-invariant subalgebra is given by
$R^W=\mathbf{k}[f_1,\dots,f_n]$. The degrees of the $f_i$
are independent of the choice of the homogeneous generators.
The degrees $d_i={\rm deg}(f_i)$ are called the \textit{characteristic degrees} of
the reflection group $W$. See \cite{springer1974regular,broue.reflextion.gps}
for a thorough exposition.

Let $W$ be the Weyl group of a compact and connected Lie group $G$,
which is finite. Then $W$ is a unitary reflection group: 
 $W$ acts on the maximal torus $T$ of $G$, and there is an induced action of $W$
on the Cartan subalgebra $\mathfrak{t}$ of the Lie algebra $\mathfrak{g}$.
The actions of $W$ on $\mathfrak{t}$ and its dual $\mathfrak{t}^\ast$ are   
faithful, so $W$ can be considered as a subgroup of $GL(\mathfrak{t}^\ast)$. Moreover,
$W$ is generated by reflections.
This action of $W$ has associated characteristic degrees $d_1,\dots,d_r$,
where $r$ is the rank of the maximal torus $T$, and it is a well-known fact
that $|W|=\prod_i d_i$. Characteristic degrees of reflection groups have many other remarkable properties,
outside the scope of this paper.

\begin{table}[ht!]
\centering
\caption{Characteristic degrees of Weyl groups $W$}
\label{table: characteristic degrees}
\begin{tabular}{llllllllllll}
Type & Lie group  & Rank & $W$ & $|W|$ & Characteristic degrees \\
\hline
$A_n$ & $SU(n+1)$  & $n\geq1$ & $\Sigma_{n+1}$ &$(n+1)!$& $2,3,...,n+1$\\
$B_n$ & $SO(2n+1)$ & $n$ & $\Z_2^n\rtimes\Sigma_n$& $n!2^n$& $2,4,\dots,2n$ \\
$C_n$ & $Sp(n)$    & $n$ & $\Z_2^n\rtimes\Sigma_n$& $n!2^n$& $2,4,\dots,2n$ \\
$D_n$ & $SO(2n)$   & $n$ & $H_n \rtimes\Sigma_n$& $n!2^{n-1}$ &$2,4,\dots,2n-2,n$ \\
$G_2$ & $G_2$      & 2   & $D_{2^2\cdot 3}$& 12 & $2,6$ \\
$F_4$ & $F_4$      & 4   & $D_{2^7\cdot 3^2}$& 1,152 & $2,6,8,12$ \\
$E_6$ & $E_6$      & 6   & $O(6,\F_2)$& 51,840 & $2,5,6,8,9,12$\\
$E_7$ & $E_7$      & 7   & $O(7,\F_2)\times \Z_2$& 2,903,040 & $2,6,8,10,12,14,18$ \\
$E_8$ & $E_8$      & 8   & $\widehat{O(8, \F_2)}$& $2^{14}3^5 5^2 7$ & $2,8,12,14,18,20,24,30$
\end{tabular}
\end{table}

As an example consider the unitary group $U(n)$ with Weyl group
the symmetric group $\Sigma_n$ on $n$ letters.   
The rank of $U(n)$ is $n$ and the $\Sigma_n$ acts on the maximal torus
$T=(S^1)^n$ by permuting the coordinates, so   
it acts on $\mathfrak{t}$
by permuting the basis vectors. 
Therefore, as a subgroup of
$GL(\mathfrak{t}^\ast)$ the Weyl group $\Sigma_n$ consists of permutation matrices.
 The invariant subalgebra is then generated by the
elementary symmetric polynomials $\epsilon_1,\dots,\epsilon_n$,   
with degrees $d_i={\rm deg}(\epsilon_i)=i$
for   
$i=1,\dots,n.$

Table \ref{table: characteristic degrees} summarizes
the Weyl groups and their associated characteristic degrees for families of
simple Lie groups, including exceptional Lie groups.
In the column for $W$, the group $H_n$ is the kernel of the multiplication map 
$\Z_2^n = \{\pm 1\}\to \Z_2$ (so $H_n$ consists of $n$--tuples containing an 
even number of $-1$'s),  $D_n$ denotes the dihedral group of order $n$,
and $\widehat{O(8, \F_2)}$ is a double cover of ${O(8, \F_2)}$.
Similar information about characteristic degrees can also be found in
\cite[p. 175]{springer1974regular} and
\cite[p. 59]{humphreys1992reflection}.

As shown in \cite{stafa.comm}, the information in   
Table \ref{table: characteristic degrees} 
and the realization of the Weyl group $W$ as a
subgroup of $GL(\mathfrak{t}^\ast)$ 
 suffice to describe the
rational cohomology of $\Comm(G)_1.$ This information will   
be
used below to describe   
the corresponding Hilbert--Poincar\'e series for $\Hom(\Z^k,G)_1$.

\subsection{Hilbert--Poincar\'e series}
Suppose $G$ has rank $r$. 
Let us denote by $A_W (q)$ 
the quantity
$$
A_W (q):=\frac{\prod_{i=1}^r (1-q^{2d_i})}{|W|},
$$
where $d_1,\dots,d_r$ are the characteristic degrees of $W$.
It was shown by Cohen, Reiner and Stafa \cite{stafa.comm} that the 
Hilbert--Poincar\'e series of
$\Comm(G)_1$ is given by the following infinite series.

\begin{thm}
Let $G$ be a compact and connected Lie group with maximal torus $T$
and Weyl group $W$. Then the Hilbert--Poincar\'e series of
the connected component of the trivial representation in $\Comm(G)$   
is given by
\begin{equation}\label{eqn: Poincere series Comm 3}
\ds P(Comm(G)_1;q,s,t) =
	A_W (q) \sum_{w\in W}
	\frac{1}{\det(1-q^2w)(1-t(\det(1+sw)-1))}.
\end{equation}
\end{thm}

Using this theorem and stable decompositions of $\Comm(G)_1$ given
above, we will   
now describe the Hilbert--Poincar\'e polynomial of the
space of ordered pairwise commuting $n$-tuples.
We begin with the following result, which  
is the fundamental step in our calculation of Poincar\'e polynomials for homomorphism spaces.

\begin{prop}\label{prop: homology of hom hat}
For $m\geqs 1$, the reduced Hilbert--Poincar\'e  
series of $\widehat{\Hom}(\Z^m,G)_1$ is given by
\begin{equation} \label{Phat}
P(\widehat{\Hom}(\Z^m,G)_1;q,s) =
A_W (q) \sum_{w\in W} \frac{(\det(1+sw)-1)^m}{\det(1-q^2 w)}.
\end{equation}
In particular, setting $s=q$ gives the reduced Poincar\'e series of $\widehat{\Hom}(\Z^m,G)_1$. 

When $m=0$, the same formulas yield the (unreduced) Hilbert--Poincar\'e and Poincar\'e series of the one-point space $\widehat{\Hom}(\Z^0,G)_1$.
\end{prop}

When $m=0$, this result asserts that the above series reduces to the constant series 1. An algebraic explanation of this fact, in terms of Molien's Theorem, is given at the end of this section.

The bigrading in this Hilbert--Poincar\'e series arises from applying the  homology isomorphism (\ref{hom-hat-iso}) described in the proof, together with  the K\"unneth Theorem. 
More specifically, let $\widehat{T^m}$ denote the $m$--fold smash product of the maximal torus $T\leqs G$ with itself. Then the coefficient of $q^i s^j$ in the above Hilbert--Poincar\'e series is the rank of the subspace $[H^i (G/T) \otimes H^j( \widehat{T^m})]^W$ of $W$--invariant elements.

\begin{proof}
First rearrange the terms in the Hilbert--Poincar\'e series of $\Comm(G)_1$:
\begin{align*}
\ds P(\Comm(G)_1;q,s,t)
	 &= A_W (q) \sum_{w\in W}\frac{1}{\det(1-q^2w)
	 		(1-t(\det(1+sw)-1))}\\
	 &= A_W (q) \sum_{w\in W}\frac{\sum_{m=0}^\infty
	 		(t(\det(1+sw)-1)))^m}{\det(1-q^2w)}\\
	 &= A_W (q) \sum_{w\in W} \sum_{m=0}^\infty \frac{
	 		(\det(1+sw)-1))^m t^m}{\det(1-q^2w)}\\
	 &= A_W (q) \sum_{m=0}^\infty \sum_{w\in W} \frac{
	 		(\det(1+sw)-1))^m t^m}{\det(1-q^2w)}\\
	 &= \sum_{m=0}^\infty \left( A_W (q) \sum_{w\in W} \frac{
	 		(\det(1+sw)-1))^m}{\det(1-q^2w)}\right) t^m.\\
\end{align*}
We claim that after setting $s=q$, the coefficient of $t^m$ in $P(\Comm(G)_1;q,s,t)$ is the 
Poincar\'e series of the stable wedge summand
$\widehat{\Hom}(\Z^m,G)_1$ appearing in the decomposition of $\Comm(G)_1$
given by Proposition~\ref{prop: X(q,G) decomposition}.

Recall that our tri-grading of the (co)homology of $\Comm(G)_1$ comes from the natural map
$$\Theta\co G/T \times_W J(T)\maps \Comm(G)_1,$$
 (see (\ref{Theta}))   
which induces isomorphisms in (rational) cohomology. 
On the left-hand side, we have
$$H^*(G/T \times_W J(T)) \isom \big(H^*(G/T)\otimes H^*(J(T))\big)^W
\isom \big(H^*(G/T)\otimes \T^*[{\widetilde{H}_*(T)}]\big)^W.
$$
Let $\T^*_m[{\widetilde{H}_*(T)]}$ denote the dual of the submodule
$$\T_m[{\widetilde{H}_*(T)}]\subset  \T[{\widetilde{H}_*(T)}]$$ 
of $m$--fold tensors. The action of $W$ preserves these submodules, so we obtain a decomposition
$$H^*(G/T \times_W J(T)) \isom \bigoplus_m \big(H^*(G/T)\otimes \T^*_m[{\widetilde{H}_*(T)}]\big)^W.$$
Note that for $m>0$, the terms in this decomposition are in fact the reduced cohomology of $G/T\cross_W \widehat{T^m}$, where $\widehat{T^m}$ denotes the $m$--fold smash product of $T$ with itself, 
so the coefficient of $t^m$ in $P(\Comm(G)_1;q,s,t)$ is the (bigraded, reduced) Hilbert--Poincar\'e series of $G/T\cross_W \widehat{T^m}$. Similarly, the $m=0$ term in  this decomposition is unreduced cohomology of $G/T\cross_W \widehat{T^0} = (G/T)/W$. Note that the rational cohomology of 
$G/T\cross_W 1\isom (G/T)/W$ is trivial, since the action of $W$ on $H^* (G/T)$ is the regular representation.

To complete the proof, it will suffice to show that  the map
\begin{equation}\label{hom-hat-iso}G/T\cross_W \widehat{T^m} \maps \widehat{\Hom}(\Z^m,G)_1\end{equation}
is an isomorphism in (rational) cohomology. As shown in the proof of \cite[Theorem 6.3]{stafa.comm}, the induced map
$$\big(G/T\cross_W \widehat{T^m} \big)/(G/T\cross_W 1) \maps\widehat{\Hom}(\Z^m,G)_1
$$
induces an equivalence in rational cohomology. But the map
$$G/T\cross_W \widehat{T^m} \maps \big(G/T\cross_W \widehat{T^m} \big)/(G/T\cross_W 1)$$ 
is also an equivalence in rational cohomology, because the rational cohomology of 
$G/T\cross_W 1\isom (G/T)/W$ is trivial, as noted above.\end{proof}
 
Baird's formula (\ref{Baird}), together with the K\"unneth Theorem, provides a bigraded Hilbert--Poincar\'e series for $\Hom(\Z^n,G)_1$, in which the coefficient of $q^i s^j$ records the rank of $[H^i (G/T)\otimes H^j(T^n)]^W$. We now compute this series.

\begin{thm}\label{thm: Poincare series of Hom}
The homology of the component of the trivial representation
in the space of commuting $n$-tuples in $G$ is given by
the following Hilbert--Poincar\'e series:  
\begin{equation*}\label{eqn: Hilb-Poincare series of Hom(Zn,G)1}
P(\Hom(\Z^n,G)_1;q,s)=A_W (q)\sum_{w\in W} 
\left(\sum_{k=0}^n {n \choose k} \frac{(\det(1+sw)-1)^k}{\det(1-q^2w)} \right).
\end{equation*}  
The Poincar\'e series of 
$\Hom(\Z^n,G)_1$ is given by
\begin{equation}\label{eqn: Poincare series of Hom(Zn,G)1}
P(\Hom(\Z^n,G)_1;q)=A_W (q) \sum_{w\in W} \frac{\det(1+qw)^n}{\det(1-q^2w)}.
\end{equation}
\end{thm}
\begin{proof} Consider the bigraded
series
\begin{equation}\label{PHomqs}
P(\Hom(\Z^n,G)_1;q,s)= \sum_{k=0}^n {n \choose k}  A_W (q)
\left(\sum_{w\in W} \frac{(\det(1+sw)-1)^k}{\det(1-q^2w)} \right),
\end{equation}
which in which the summands are the Hilbert--Poincar\'e series from Proposition \ref{prop: homology of hom hat}. Since the terms in this sum 
match  the terms from
the stable decomposition of the space $\Hom(\Z^n,G)_1$
in equation (\ref{eqn: stable decomp Hom}), we see that setting
$q=s$ in (\ref{PHomqs}) yields  the Poincar\'e series of 
$\widehat{\Hom}(\Z^m,G)_1$:
\begin{align*}
P(\Hom(\Z^n,G)_1;q)= &A_W (q)\sum_{k=0}^n {n \choose k} 
\left(\sum_{w\in W} \frac{(\det(1+qw)-1)^k}{\det(1-q^2w)} \right)\\
=& A_W (q)
 \sum_{w\in W} \frac{\sum_{k=0}^n {n \choose k} (\det(1+qw)-1)^k}{\det(1-q^2w)}.
\end{align*}
Setting $x = \det(1+qw)-1$ in the binomial expansion
$$(1+x)^n = \sum_{k=0}^n  {n \choose k} x^k$$
 gives
$$\sum_{k=0}^n {n \choose k} (\det(1+qw)-1)^k  = \det(1+qw)^n,$$
yielding the simplified form (\ref{eqn: Poincare series of Hom(Zn,G)1}).

Finally, we check that the bigrading in $P(\Hom(\Z^n,G)_1;q,s)$ agrees with the bigrading arising from the K\"unneth Theorem applied to (\ref{Baird}). More precisely, we want to show that for each $i, j\geqs 0$,
\begin{equation}\label{qisj} \rk [H^i (G/T) \otimes H^j (T^n)]^W 
= \sum_{k=0}^n {n\choose k} \rk [H^i (G/T) \otimes H^j (\widehat{T^k})]^W.
\end{equation}
The spaces $G/T \cross T^n$ form a simplicial space $G/T \cross T^\bullet$ as $n$ varies, where the simplicial structure arises from the bar construction on $T$ (so the face and degeneracy maps are the identity on the $G/T$ factors). The main result of~\cite{adem.cohen.gitler.bahri.bendersky} provides a stable splitting of the spaces
$G/T\cross T^n$:
$$\gamma\co \Susp (G/T \cross T^n)
\srm{\heq}
\bigvee_{k=0}^n \bigvee_{{n\choose k}} \Susp (G/T \cross \widehat{T^k}).$$
 In fact, the stable splittings from~\cite{adem.cohen.gitler.bahri.bendersky} apply to any (sufficiently nice) simplicial space, and are natural with respect to simplicial maps. In particular, we can consider the simplicial maps   
$$G/T \longleftarrow G/T \cross T^\bullet \maps  T^\bullet,$$
where $G/T$ is viewed as a constant simplicial space, and $ T^\bullet$ is the bar construction on $T$. Naturality of the splittings yields a commutative diagram
\begin{center}
\begin{tikzcd}
\Susp (G/T)  & \ds{\bigvee_{k=0}^n \bigvee_{{n\choose k}} \Susp (G/T \cross \widehat{T^k})}\arrow{l} \arrow{r} &\ds{\bigvee_k \bigvee_{{n\choose k}} \Susp \widehat{T^k}}\\
\Susp (G/T) \arrow{u}& \Susp (G/T \cross T^n) \arrow{r}  \arrow{l}   \arrow{u}{\gamma} &
		\Susp T^n, \arrow{u}\\
\end{tikzcd}
\end{center}
in which the vertical maps are weak equivalences. Commutativity 
implies that the map on cohomology induced by $\gamma$ respects the K\"unneth decompositions of 
$H^p  (\Susp (G/T \cross T^n))$ and  $H^p (\Susp (G/T \cross \widehat{T^k}))$ ($p\geqs 1$), so that 
$\gamma$ induces isomorphisms
\begin{equation}\label{gamma*}\bigoplus_{k=0}^n  \bigoplus_{{n\choose k}} (H^i(G/T)\otimes H^j( \widehat{T^k}))
\srm{\isom} H^i (G/T)\otimes H^j ( \widehat{T^k})\end{equation}
for each $i,j\geqs 0$.
Moreover, $W$ acts simplicially on $G/T\cross T^\bullet$, so naturality implies that $\gamma$ is $W$--equivariant, and hence the maps (\ref{gamma*}) induce isomorphisms when restricted to  $W$--invariants.
This establishes the desired equality (\ref{qisj}).
\end{proof}

\begin{rmkMT} Since $\Hom(\Z^n,G)_{1}$ is path connected, the constant term in its Poincar\'e series must be 1. This can be understood in terms of a classical theorem of Molien \cite{molien1897} 
(also see \cite[p. 289]{shephard1954finite}).
Let $R=\mathbf{k}[x_1,\dots,x_r]$ and $W$ be as above, with $x_1,\dots,x_r$ in degree 1.
 Molien's Theorem  states that the number of 
linearly independent elements in degree $m$ in the invariant ring 
$R^W=\mathbf{k}[x_1,\dots,x_r]^W$
is given by the coefficients of the generating function
$$
\sum_{m=0}^{\infty} l_m q^m = \frac{1}{|W|}\sum_{w\in W} \frac{1}{\det(1-qw)} .  
$$
Moreover,   
Chevalley \cite{chevalley1955invariants} 
and Shephard--Todd \cite{shephard1954finite} give the following 
generating function for $R^W=\mathbf{k}[f_1,\dots,f_r]$ 
$$\sum_{m=0}^{\infty} l_m q^m =\prod_{i=1}^r \frac{1}{(1-q^{d_i})},$$
Therefore, after doubling the degree of $q$ one obtains the equation
$$1 =\frac{\prod_{i=1}^r (1-q^{2d_i})}{|W|} \sum_{w\in W} \frac{1}{\det(1-q^2w)},$$  
which corresponds to the constant term in the Hilbert--Poincar\'e series
of the spaces of homomorphisms $\Hom(\Z^n,G)_{1}$ in Theorem \ref{thm: Poincare series of Hom}.
\end{rmkMT}

\begin{rmkH1}
By work of Gomez--Pettet--Souto~\cite{gomez.pettet.souto},
$$\pi_1 (\Hom(\Z^n, G)_1) \cong (\pi_1 G)^n.$$  
It follows that
\begin{equation}\label{GPS}\rk (H^1 (\Hom(\Z^n, G)_1)) = n \cdot \rk (H^1 G).\end{equation}
This can in fact be seen directly from the formula in Theorem~\ref{thm: Poincare series of Hom} by analyzing the coefficient of $q$. Indeed, any non-zero coefficient of $q$ must come from one of the terms $\det (1+qw)^n$.  We have 
$$\det (1+qw) = \prod (1+ \lambda(w)q),$$
where $\lambda(w)$ ranges over the eigenvalues of $w$ (counted with multiplicity). Hence 
the constant term of $\det (1+qw)$ is 1, and the coefficient of $q$ is the trace of $w$ (acting on $\mathfrak{t}^*$).
It follows that the coefficient of $q$ in $P(\Hom(\Z^n, G)_1; q)$ is precisely
$$\frac{n}{|W|} \sum_{w\in W} \Tr (w) = n \langle \chi, 1\rangle = n\cdot \rk ((\mathfrak{t}^*)^W),$$
where $\chi$ is the character of the representation of $W$ on $\mathfrak{t}^*$ and $\langle \chi, 1\rangle$ is the inner product of this character with the trivial 1-dimensional character.
Since this representation is isomorphic to the natural representation of $W$ on $H^1 (T; \C)$, we find that $\rk (H^1 (\Hom(\Z^n, G)_1)) = n\cdot \rk (H^1 (T; \C)^W)$. 
As discussed above, $H^*(G) \isom [H^\ast(G/T) \otimes H^\ast(T)]^W$, and
 the action of $W$ on $H^*(G/T)$ is the regular representation. Hence 
$$H^1 (G) \isom (H^1(T))^W,$$
and combining the previous two formulas yields (\ref{GPS}).
\end{rmkH1}

\section{Poincar\'e series of $\X(m,G)_1$}\label{sec: Poincare series of X(q,G)}

The following theorem describes the  Poincar\'e
series of $\X(m,G)_1$ for all $m \geq 2.$
Note that when $G$ is a compact connected Lie group, it follows from
 Bergeron--Silberman~\cite{bergeron2016note} that $\X(m,G)_1 = \Comm(G)_1$.

\begin{thm} Let $G$ be a reductive connected Lie group.
Then the natural inclusion maps
$$X(2,G)_1 \hookrightarrow X(3,G)_1 \hookrightarrow 
			\cdots \hookrightarrow X(m,G)_1 \hookrightarrow \cdots$$
all induce homotopy equivalences after one suspension.
In particular, the Hilbert--Poincar\'e series of
$\X(m,G)_1$, for all $m\geq 2$, is given by
\begin{equation}\label{eqn: Poincere series Comm 4}
P(\X(m,G)_1;q,s,t) =
	A_W (q) \sum_{w\in W}
	\frac{1}{\det(1-q^2w)(1-t(\det(1+sw)-1))},
\end{equation}
where $W$ is the Weyl group of a maximal compact subgroup $K\leqs G$.
\end{thm}
\begin{proof}
By Proposition~\ref{prop: X(q,G) decomposition}, there is a stable 
decomposition of $\X(m,G)_1$
into a wedge sum
\begin{equation}\label{q-s}
\Sigma \X(m,G)_1 \simeq \Sigma
	\bigvee_{k \geq 1} \widehat{\Hom}(F_k/\Gamma^m_k,G)_1.
\end{equation}
Consider the commutative diagram of cofibrations
\begin{center}
\begin{tikzcd}
S_{n,2}(G) \arrow{d} \arrow{r} & \Hom(\Z^n,G)_{1} \arrow{d}{i} \arrow{r} &\widehat{\Hom}(\Z^n,G)_{1} \arrow{d}\\
S_{n,m}(G)  \arrow{r} & \Hom(F_n/\Gamma^m_n,G)_{1}\arrow{r} &
		\widehat{\Hom}(F_n/\Gamma^m_n,G)_{1},\\
\end{tikzcd}
\end{center}
where $S_{n,m}(G)$ is the subspace of $\Hom(F_n/\Gamma^m_n,G)_{1}$
consisting of $n$-tuples with at least one coordinate the identity,
and $m \geq 2$. 
The middle vertical map
$$i\co \Hom(\Z^n,G)_{1} \hookrightarrow  \Hom(F_n/\Gamma^m_n,G)_{1}$$
is a homotopy equivalence: by \cite{bergeron}, up to homotopy we can replace $G$ by a maximal compact subgroup, and by~\cite{bergeron2016note}, the map $i$ is a homeomorphism in the compact case.
The first vertical map
$$S_{n,2}(G) \hookrightarrow S_{n,m}(G)$$
is a homotopy equivalence by the Gluing Lemma
\cite{RBrown}, since these spaces can be built up inductively 
as pushouts of subspaces of the form 
$$\{(g_1, \ldots, g_n) \,:\, g_i =1 \textrm{ for all } i\in I\}$$
for various $I \subset \{1, \ldots n\}$, and on these subspaces the 
results from  \cite{bergeron} and \cite{bergeron2016note} apply.
Applying the Gluing Lemma again, the third vertical map
$$\widehat{\Hom}(\Z^n,G)_{1} \to \widehat{\Hom}(F_n/\Gamma^m_n,G)_{1}$$
is a homotopy equivalence as well, and the theorem follows from 
the decompositions (\ref{q-s}). 
\end{proof}

\section{Ungraded cohomology and $K$--theory}\label{ungraded-sec}

The \textit{ungraded cohomology} $H^u (X;R)$  of a space $X$
with coefficients in $R$ refers to the ungraded direct sum of all the cohomology groups
of $X$ as an $R$--module.
It is a classical result \cite{borel1953cohomologie}
that the ungraded cohomology of $G/T$ with rational coefficients, viewed as a $W$--module, is the
regular representation $\Q W$. This alone yields some interesting consequences.
Let $M$ be a graded $\Q W$-module. Then it follows that
$(H^u(G/T;\Q)\otimes M)^W \isom M.$ Applying this principle to  formulas (\ref{Baird}) and (\ref{Comm})
yields the following result.

\begin{prop}
Let $G$ be a compact and connected Lie group. Then
\begin{enumerate}
\item the ungraded rational cohomology of the compact and connected
Lie group $G$ is the same as the ungraded rational cohomology
of its maximal torus:
$$(H^u(G/T;\Q)\otimes H^u(T))^W\isom H^u(T);$$

\item the ungraded cohomology of $\Hom(F_n/\Gamma^m_n,G)_1$ is given by
$$H^u(\Hom(F_n/\Gamma^m_n,G)_1;\Q) \isom H^u(T^n;\Q)$$
for all integers  
$m\geq 2.$

\end{enumerate}
\end{prop}

It is quite interesting, although not a surprise, that the
maximal torus $T\subset G$ plays a fundamental role in the
topology of nilpotent representations into $G$, similar
to the role it plays in the topology
of $G$ from the classical theory of Lie groups.
Recall that the rational cohomology of Lie groups can
be described by the cohomology of a product of as many  spheres 
of odd dimension as the rank of $G$ \cite{reeder1995cohomology}.
It would however be very compelling to   
understand, in a topological manner, the
regrading process that produces the cohomology of $\Hom(\Z^n,G)_1$ and
$\Comm(G)_1$ from the cohomology of $T^n$ and $J(T)$, respectively.

\begin{cor}\label{tot-rk}
Let $G$ be a compact and connected Lie group of rank $r$. Then
\begin{align*}
\sum_{k\geq 0} \rk ( H^k(\Hom(\Z^n,G)_1;\Q))=
\sum_{k\geq 0} \rk ( H^k(T^n;\Q) )= 2^{nr}.
\end{align*}
\end{cor}

\begin{rmk}
Having identified the total rank of the cohomology of these spaces, 
one can ask if they satisfy Halperin's Toral Rank Conjecture, which states that 
if a topological space $X$ has an almost free action of a torus of rank $k$, 
then the rank of the total cohomology of $X$ is at least $2^k$; 
see \cite[Problem 1.4]{halperin1985}. In this setting, the conjecture 
predicts that if a torus $T'$ acts almost freely on the space of homomorphisms
$\Hom(\Z^n,G)_1$, then the rank of $T'$ must be at least $nr.$ Hence it would 
be interesting to understand almost-free torus actions on these spaces.
\end{rmk}

Having identified the total rank of the cohomology of $\Hom(\Z^n,G)_1$, 
we can also identify its rational complex $K$--theory.

\begin{cor} \label{K-cor}
Let $G$ be a compact and connected Lie group of rank $r$. Then
\begin{align*}
\rk ( K^i(\Hom(\Z^n,G)_1))= 2^{nr-1}.
\end{align*}
for every $i$.
\end{cor}

\begin{proof} Since $\Hom(\Z^n,G)_1$ is a finite CW complex, the Chern character provides an isomorphism from $K^i(\Hom(\Z^n,G)_1)\otimes \bbQ$   
to the sum of the rational cohomology groups of $\Hom(\Z^n,G)_1$ in dimensions congruent to $i$ (mod 2). 
The fibration sequence
$$G/T \times T^n \maps G/T \times_W T^n \maps BW$$
implies that the Euler characteristic of $G/T \times_W T^n$ is zero, and the same follows for
$\Hom(\Z^n,G)_1$ by Baird's result (\ref{Baird}). Hence 
\begin{align*}
\sum_{k  \textrm{ even}}  \rk ( H^k(\Hom(\Z^n,G)_1) )= 
				\sum_{k  \textrm{ odd}}  \rk ( H^k(\Hom(\Z^n,G)_1)),  
\end{align*}
and the result follows from Corollary~\ref{tot-rk}.
\end{proof}

\begin{question} When $G$ is a product of groups of the form $SU(r)$, $U(q)$, and $Sp(k)$, Adem and Gomez~\cite[Corollary 6.8]{AG-equivK} showed that the $G$--equivariant $K$--theory ring $K^*_G (\Hom(\Z^n, G))$ is free as module of rank $2^{nr}$ over the representation ring $R(G)$ (for such $G$ we have $\Hom(\Z^n, G)_1 = \Hom(Z^n, G)$).
In light of Corollary~\ref{K-cor}, it is natural to ask whether the map $R(G) \otimes K^*(\Hom(\Z^n,G)) \to K_G^* (\Hom(\Z^n,G))$ is an isomorphism for these groups.
\end{question}

\begin{rmk} We now explain how to compute $\rk (H^u(\Hom(\Z^n,G)_1))$, and the Euler characteristic $\chi(\Hom(\Z^n,G)_1)$, directly from Theorem~\ref{thm: Poincare series of Hom}, by setting $q=1$ or $-1$ in the formula
\begin{equation}\label{p}P(\Hom(\Z^n,G)_1; q) = 
A_W (q) \sum_{w\in W}  \frac{\det(1+qw)^n}{\det(1-q^2w)}.
\end{equation}
To do so, we must compute the multiplicities of $\pm 1$ as roots of $A_W (q)$ and of $\det (1-q^2 w)$ ($w\in W)$. We have 
$$A_W (q) = \frac{1}{|W|}\prod_{i=1}^r (1-q^{d_i})(1+q^{d_i}),$$
so the multiplicity of $\pm 1$ as a root of $A_W (q)$ is $r = \rk (G)$. On the other hand, 
$$\det (1-q^2 w) = \prod_i (1-q^2 \lambda_i (w))^{n_i},$$
where the numbers $\lambda_i (w)$ are the eigenvalues of $w$ (acting on $\mathfrak{t}^\ast$) and $n_i$ is the dimension of the corresponding 
eigenspace. So the multiplicity of $\pm 1$ is the dimension of the eigenspace for $\lambda_i (w) = 1$, which is strictly less than $\rk (G)$ unless $w = 1$, in which case it is exactly $\rk (G)$. Canceling factors of $1\pm q$ in 
$$\frac{\prod_{i=1}^r (1-q^{d_i})(1+q^{d_i})}{\det(1-q^2w)}\det(1+qw)^n$$
and plugging in $q=\pm 1$, we see that all terms for $w\neq 1$ are zero.

Now consider what happens when we plug in $q=-1$ into (\ref{p}). The term for $w=1$ contains the determinant of $I +qI= I-I = 0$ as a factor, so it too vanishes.
This gives another proof that $\chi (\Hom(\bbZ^n, G)_1) = 0$.

To calculate $\rk (H^u (\Hom(\bbZ^n, G)_1))$, we must analyze the $w=1$ term of (\ref{p}) more closely. This term has the form
\begin{align*}\frac{\prod_{i=1}^r (1+q^{d_i})(1-q^{d_i})}{|W|} \cdot & \frac{\det((1+q)I)^n}{\det((1-q^2)I)}\\
=  &\frac{\prod_{i=1}^r (1+q^{d_i})(1-q^{d_i})}{|W|}  \cdot \frac{(1+q)^{rn}}{(1-q^2)^r}\\
= &\frac{\prod_{i=1}^r (1+q^{d_i})(1+q+q^2 +\cdots +q^{d_i-1})}{|W|} \cdot \frac{(1+q)^{rn}}{(1+q)^r}.
\end{align*}
Plugging in $q=1$, we find that
$$\rk (H^u (\Hom(\bbZ^n, G)_1)) = 
\frac{\left(\prod_{i=1}^r 2d_i  \right) 2^{rn}}{|W| \cdot 2^r}
=2^{rn},$$
where we have used the equation $\prod_{i=1}^r d_i = |W|$. 
\end{rmk}

\section{Examples of Hilbert--Poincar\'e series}\label{sec: examples Poincare}

Using Theorem \ref{thm: Poincare series of Hom} and Table
\ref{table: characteristic degrees}, one can obtain explicit 
formulas for the Hilbert--Poincar\'e and Poincar\'e series described in this article.
We demonstrate this for some low-dimensional
Lie groups and for the exceptional Lie group $G_2$. We give only the formulas
for the Poincar\'e series. The Hilbert--Poincar\'e series can then
be deduced similarly from Theorem \ref{thm: Poincare series of Hom} and
are left to the reader. 

\begin{ex}[$G=SU(2)$]
The maximal torus of $SU(2)$ has rank 1
and the Weyl group is isomorphic to $W=\Z_2$. 
The dual space $\mathfrak{t}$ is 1-dimensional, and $W$ is represented as
$\{1,-1\} \subset GL(\mathfrak{t}^\ast)$.  
The only characteristic degree of $W$ is $d_1=2$.
Therefore, we have
$A_W (q)=(1-q^4)/2$ and $\det(1+qw)$ equals $1+q$ and $1-q$, 
for $w$ equal to 1 and -1, respectively, and
$$
\frac{\det(1+qw)^n}{\det(1-q^2w)}=
\begin{cases}
      \ds\frac{(1+q)^n}{1-q^2} 	& \mbox{if $w=1$}, \\
      		\\
      \ds\frac{(1-q)^n}{1+q^2} 		& \mbox{if $w=-1$}. \\
\end{cases}
$$
We know the space of commuting $n$-tuples in $SU(n)$ is path connected,
so  
it equals the component of the trivial representation.
\begin{align*}
P(\Hom(\Z^n,SU(2));q)&= A_W (q) \sum_{w\in \Z_2}  \frac{\det(1+qw)^n}{\det(1-q^2w)}\\
	&= \frac{1}{2}\bigg((1+q)^n (1+q^2) + (1-q)^n(1-q^2)\bigg), 
\end{align*}
which agrees with calculations in \cite{bairdcohomology}. 
\end{ex}

\begin{ex}[$G=U(2)$]
The maximal torus of $U(2)$ has rank 2, the Weyl group
$W \isom \Z_2$ acts on $\mathfrak{t}^*$ via the matrices   
$$ \left\{\left(
\begin{array}{cc}
1 & 0 \\
0 & 1 \\
\end{array}
\right), \left(
\begin{array}{cc}
0 & 1 \\
1 & 0 \\
\end{array}
\right)\right\},$$
and the characteristic degrees of $W$   are $d_1=1,d_2=2$.
We know the space of commuting $n$-tuples in $U(n)$ is path connected,
so again it equals the component of the trivial representation. We have
$A_W (q)=(1-q^2)(1-q^4)/2$ and
$$
\frac{\det(1+qw)^n}{\det(1-q^2w)}=
\begin{cases}
      \ds\frac{(1+q)^{2n}}{(1-q^2)^2} 	& \mbox{if $w=1$}, \\
      		\\
      \ds\frac{(1-q^2)^n}{1-q^4} 			& \mbox{if $w\neq 1$}. \\
\end{cases}
$$
Therefore, we get the following 
Poincar\'e series
$$
P(\Hom(\Z^n,U(2));q)=\dfrac{1}{2}\bigg((1+q)^{2n} (1+q^2) + (1-q^2)^{n+1}\bigg).$$
For example, we get
\begin{align*}
P(\Hom(\Z^2,U(2));q)&=	1+2q+2q^2+4q^3+5q^4+2q^5,\\
P(\Hom(\Z^3,U(2));q)&= 1 + 3q + 6q^2 + 13q^3 + 18q^4 +13q^5 + 6q^6 + 3q^7 + q^8,\\
P(\Hom(\Z^4,U(2));q)&= 1 + 4q + 12q^2 + 32q^3 + 54q^4 +56q^5 + 44q^6 + 32q^7\\
&\,\,\,\,\,\,\,\,\,\,\, + 17q^8+4q^9.
\end{align*}
\end{ex}

\begin{ex}[$G=U(3)$] The maximal torus has rank 3 and the Weyl group is
the symmetric group on 3 letters
$$W=\Sigma_3=\{e,(12),(13),(23),(123),(132)\}.$$
The characteristic degrees of $W$ are $1$, $2$, and $3$, so 
$$A_W (q) = \frac{1}{6}(1-q^2)(1-q^4)(1-q^6).$$
The matrix representations $W \leqslant GL(\mathfrak{t}^\ast)$ can
be obtained by applying each permutation in $\Sigma_3$ to the
rows of the $3\times 3$ identity matrix $I_{3\times 3}$. This can be
done in general for the Weyl group $\Sigma_n$ of $U(n).$
For the transpositions $w=(12), (13), (23)\in W$ and for the 3-cycles we obtain the
same determinants, respectively, since they are in the same conjugacy class.
Hence we get
$$\ds
\frac{\det(1+qw)^n}{\det(1-q^2w)}=
\begin{cases}
      \ds\frac{(1+q)^{3n}}{(1-q^2)^3} 	& \mbox{if $w=e$}, \\
      		\\
      \ds\frac{(1+q)^n(1-q^2)^n}{(1-q^2)(1-q^4)} 	& \mbox{if $w=(12), (13), (23)$}, \\
      		\\
      \ds\frac{(1+q^3)^n}{1-q^6} 			& \mbox{if $w=(123),(132)$}. \\
\end{cases}
$$
Therefore, the Poincar\'e series is given by
\begin{align*}
P(\Hom&(\Z^n,U(3)),q)  
	=\frac{1}{6} \bigg(   
	(1+q^2)(1+q^2+q^4)(1+q)^{3n} \\ &+
	3 (1-q^6)(1+q)^n(1-q^2)^n +
	2 (1-q^2)(1-q^4)(1+q^3)^n
	\bigg).
\end{align*}
In particular, the following are the Poincar\'e series for pairwise commuting pairs,
triples, and quadruples in $U(3)$, respectively:
\begin{align*}
P(\Hom(\Z^2,U(3)),q) = & 1+2q+2q^2+ 4q^3+7q^4+10q^5+11q^6+8q^7+8q^8\\		
		&+8q^9+3q^{10}\\
P(\Hom(\Z^3,U(3)),q)=  & \,1 +3q + 6q^2 + 14q^3 + 30q^4 + 54q^5 + 73q^6 + 75q^7 + 75q^8\\
				& \,\,\,\,+ 73q^9 + 54q^{10} + 30q^{11} + 14q^{12} + 6q^{13} + 3q^{14} + q^{15}, \\
P(\Hom(\Z^4,U(3)),q)= & \,1 +4q + 12q^2 + 36q^3 + 96q^4 + 212q^5 + 357q^6 + 472q^7  \\
				&  \,\,\,\,+555q^8+ 604q^9 + 574q^{10} + 468q^{11} + 330q^{12} + 204q^{13} \\
				&  \,\,\,\,+ 113q^{14} + 48q^{15} + 10q^{16}.
\end{align*}
\end{ex}

\begin{ex}[$G=G_2$] Now consider the exceptional Lie group $G_2$,
a 14 dimensional submanifold of $SO(7)$, which has rank 2 
and Weyl group the dihedral group $W=D_{12}$ of order 12 with presentation
$\langle s,t | s^2,t^6,(st)^2 \rangle.$ We can write
$W=\{1,t,t^2,t^3,t^4,t^5,s,st,st^2,st^3,st^4,st^5\}$ as a subgroup
of $GL(\mathfrak{t}^\ast)$ by setting
$$s=\left(
\begin{array}{cc}
1 & 0 \\
0 & -1 \\
\end{array}
\right),
\text{ and }
t=\frac{1}{2}\left(
\begin{array}{cc}
1 & \sqrt{3} \\
-\sqrt{3} & 1 \\
\end{array}
\right).$$
The characteristic degrees of $W$ are 2 and 6 as given in
Table \ref{table: characteristic degrees}.
The space $\Hom(\Z^n,G_2)$ is not path-connected since it has an
elementary abelian 2--subgroup of rank 3, which is non-toral.
Setting $t=1$ and $q=s$ the Poincar\'e series of $\Hom(\Z^n,G_2)_1$ is calculated
using Equation \ref{eqn: Poincare series of Hom(Zn,G)1}:
\begin{align*}
P(& \Hom(\Z^n,G_2)_1;q)= 1 + \frac{1}{12}
\big[(2{q}^{14}-2{q}^{12}-2{q}^{2}+2)(-{q}^{2}+1)^{n-1}\\
&+ (2{q}^{12}-2{q}^{10}-2{q}^{8}+4{q}^{6}-2{q}^{4}-2{q}^{2}+2)({q}^{2}-q+1)^{n}\\
&+ (2{q}^{12}+2{q}^{10}-2{q}^{8}-4{q}^{6}-2{q}^{4}+2{q}^{2}+2)( {q}^{2}+q+1)^{n}\\
&+ ({q}^{12}-2{q}^{10}+2{q}^{8}-2{q}^{6}+2{q}^{4}-2{q}^{2}+1)(-1+q)^{2n}
+(-4{q}^{12}+4)(-{q}^{2}+1)^{n} \\
&+ ({q}^{12}+2{q}^{10}+2{q}^{8}+2{q}^{6}+2{q}^{4}+2{q}^{2}+1)(q+1)^{2n}\big].
\end{align*}
For example, for $n=1,2,3$ we obtain:
\begin{align*}
P(\Hom(\Z^1,G_2)_1;q)=&1+q^3+q^{11}+q^{14}=P(G_2;q),\\
P(\Hom(\Z^2,G_2)_1;q)=&1+q^2+2q^3+q^4+2q^5+q^6+q^{10}+2q^{11}+2q^{13}+3q^{14},\\
P(\Hom(\Z^3,G_2)_1;q)=&1+3q^2+3q^3+6q^4+9q^5+3q^6+3q^7+3q^8+2q^9+3q^{10}\\
& \,\,\,+3q^{11}+3q^{12}+9q^{13}+6q^{14}+3q^{15}+3q^{16}+q^{18}.
\end{align*}
\end{ex}

It can be observed from the above formula for   
the Poincar\'e series that
the \textit{rational homological dimension} of the spaces of commuting $(2k-1)$-tuples in $G_2$ 
is the same as that for commuting $2k$-tuples, namely $12+4k$.   
However, it is not clear if there is a topological
reason for this phenomenon.

\begin{rmk} The above formulas suggests that for odd $n$,
$\Hom(\Z^n,G)_1$ is a \e{rational Poincar\'e duality space}; in particular,
the coefficients of the above Poincar\'e series are palindromic for $n$ odd. 
A geometric proof of this fact was provided to us by Antol\'in, Gritschacher, and Villarreal (private communication). 
Briefly, Baird's theorem (as discussed in Section~\ref{sec: topology Hom}) reduces us to showing that for $n$ odd, 
the action of $W$ on the manifold $G/T \cross T^n$ is \e{orientation preserving}. From the case $n=1$, where $H^*(G)$ is known, 
we can see that for each $w\in W$ the action of $w$ on $G/T$ is orientation-preserving if and only if the action of $w$ on $T$ is 
orientation preserving. Since $W$ acts diagonally on $T^n$, the result follows.
\end{rmk}

\end{document}